\documentclass[11pt]{amsart}

\usepackage{amsmath}
\usepackage{amssymb}
\usepackage{mathtools}
\usepackage[T1]{fontenc}
\usepackage[english]{babel}
\usepackage[hidelinks]{hyperref}
\usepackage[margin=1.15in]{geometry}

\newtheorem{thm}{Theorem}[section]
\newtheorem{lem}[thm]{Lemma}
\newtheorem{prop}[thm]{Proposition}
\newtheorem{cor}[thm]{Corollary}

\theoremstyle{definition}
\newtheorem{defn}[thm]{Definition}
\newtheorem{example}[thm]{Example}

\newtheorem{remark}[thm]{Remark}

\newcommand{\R}{\mathbb{R}}

\newcommand{\ip}[2]{\langle #1,\, #2 \rangle}

\newcommand{\bn}[2]{{\textstyle\binom{#1}{#2}}}

\DeclareMathOperator{\Tr}{Tr}
\DeclareMathOperator{\Sym}{Sym}

\DeclareMathOperator{\diag}{diag}

\DeclareMathOperator{\PG}{PG}

\begin{document}

\title[Finite free convolution via reproducing kernels]{Finite free convolution via reproducing kernels \\
and squarefree algebras}

\author[T.~Sinclair]{Thomas Sinclair}
\address{Department of Mathematics, Purdue University\\ 150 N.\ University St.\\ West Lafayette, IN 47907}
\email{tsincla@purdue.edu}
\thanks{The author was partially supported by NSF grant DMS-2055155.}

\date{\today}

\begin{abstract}
    We give a structural account of the finite free convolutions of Marcus, Spielman, and Srivastava in terms of reproducing kernel inner products on polynomial spaces and a multilinear model over the squarefree algebra. In this model, additive convolution becomes algebra multiplication, and the nilpotent logarithm linearizes it, recovering the finite free cumulants of Arizmendi and Perales. This perspective leads to a class $\mathcal{LC}_n$ of multilinear polynomials characterized by nonpositivity of higher-order cumulants, closed under additive convolution and satisfying several key permanence properties associated with negatively dependent measures. We show that every graph Laplacian pencil belongs to this class, with higher-order cumulants given by Hamiltonian cycle counts in induced subgraphs.
\end{abstract}

\maketitle

\section{Introduction}\label{sec:intro}

Finite free convolution, introduced by Marcus, Spielman, and
Srivastava~\cite{MSS15,MSS22}, plays a central role in the theory
of interlacing polynomials and the study of real stability and real-rootedness properties of polynomials. While the theory was from the inception grounded in part by Voiculescu's theory of free probability and free convolutions \cite{Voi86, Voi92}, the subject has rapidly developed into a fascinating bridge between the theory of real polynomials, combinatorics, and operator algebraic methods \cite{AFPU, AP18}. Although the additive and multiplicative finite free convolutions share many formal similarities, they are typically presented through distinct combinatorial formulas, and the structural relationship between them is not immediately apparent.

The starting point of this paper is the observation that both operations arise naturally from a single reproducing kernel.  The apolar kernel
\[
K_n(\xi,t)=(1+\xi t)^n
\]
gives rise to the finite free multiplicative and additive convolutions through pushing the kernel forward under the multiplication and addition maps, respectively.  In the resulting coordinates, multiplicative convolution becomes coefficientwise multiplication, while additive convolution becomes binomial convolution.
Moreover, the underlying structure becomes considerably simpler after
passing to a multilinear model.  Let
\[
\mathcal A_n= \mathbb R[t_1,\dots,t_n]/(t_1^2,\dots,t_n^2)
\]
denote the squarefree algebra.  Under multilinear symmetrization, the finite free convolutions are transported to operations on \(\mathcal A_n\).  In this setting multiplicative convolution becomes the Hadamard product, while additive convolution becomes ordinary squarefree algebra multiplication.

Once additive convolution is identified with multiplication in a nilpotent algebra, it becomes natural to study its logarithmic coordinates.  The nilpotent logarithm linearizes additive convolution and produces a family of multilinear cumulants satisfying
\[
    \kappa_S(f\boxplus g)
    =\kappa_S(f)+\kappa_S(g).
\]
After a simple rescaling, these cumulants recover the finite free
cumulants introduced by Arizmendi and Perales~\cite{AP18}.  From
this perspective, finite free cumulants arise not as an auxiliary
construction but as the ordinary logarithm of the squarefree algebra.

The logarithmic viewpoint suggests a natural class \(\mathcal{LC}_n\) of multilinear polynomials whose higher-order cumulants are nonpositive.  We show that this class is closed under additive convolution, external fields, conditioning, and marginalization, and that its coefficient functions are log-submodular.  These are among the characteristic permanence properties appearing in the theory of negatively dependent measures~\cite{Pem00,BBL09}.

Finally, we illustrate the theory through determinantal examples.  For graph Laplacian pencils, the trace--log expansion admits a striking combinatorial interpretation: the cumulants are given by Hamiltonian cycle counts in induced subgraphs.  As a consequence, every graph Laplacian pencil belongs to \(\mathcal{LC}_n\).  We also discuss related constructions arising from edge Laplacians.

The paper is organized as follows.
Section~\ref{sec:apolar} introduces the apolar inner product, its reproducing kernel, and the $\beta_n$-transform, and derives kernel-level formulas for both convolutions.
Section~\ref{sec:multilinear} develops the multilinear model, showing that additive convolution is algebra multiplication and that symmetrization recovers the univariate theory.
Section~\ref{sec:cumulants} introduces the nilpotent logarithm, proves cumulant linearization, and establishes the connection to the finite free cumulants of Arizmendi--Perales~\cite{AP18}.
Section~\ref{sec:logconcave} defines $\mathcal{LC}_n$ and proves closure under additive convolution.
Section~\ref{sec:laplacian} derives the trace--log formula, computes the cumulants of Laplacian pencils, and develops the edge Laplacian and example.
The appendix records a projective-geometric interpretation of the coefficients appearing in the $q$-additive finite free convolution.

\section{Apolar Kernels and Finite Free Convolution}
\label{sec:apolar}

Let $\R_n[t]$ denote the space of real polynomials of degree at most~$n$.  The \emph{apolar inner product} is defined on monomials by
\[
  \ip{t^j}{t^k}_{\mathrm{ap}}
  := \bn{n}{k}^{-1}\, \delta_{jk},
\]
or equivalently, for
$p(t) = \sum_{k=0}^n p_k t^k$ and
$q(t) = \sum_{k=0}^n q_k t^k$,
\[
  \ip{p}{q}_{\mathrm{ap}}
  = \sum_{k=0}^n \bn{n}{k}^{-1}\, p_k\, q_k.
\]
Under the inversion
\begin{equation}\label{eq:star-inversion}
    q(t)\mapsto q^\star(t) := (-t)^n q(-1/t)
\end{equation}
we have that $\ip{p}{q^\star}_{\mathrm{ap}}$ agrees with the standard apolar pairing as in \cite[Chapter IV]{Marden} or \cite{Rez92}; see also \cite{Leake} for a modern treatment of apolarity in the context of real stability.

\begin{prop}\label{prop:reproducing-kernel}
The function $K_n(\xi,t) := (1+\xi t)^n$ is the reproducing kernel for $(\R_n[t],\ip{\cdot}{\cdot}_{\mathrm{ap}})$. That is, for all $p\in\R_n[t]$,
\[
  p(\xi) = \ip{p(t)}{K_n(\xi,t)}_{\mathrm{ap}}.
\]
\end{prop}

\begin{proof}
Expanding and using orthogonality,
\[
  \ip{p(t)}{(1+\xi t)^n}_{\mathrm{ap}}
  = \sum_{k=0}^n p_k \cdot \bn{n}{k}\,\xi^k
  \cdot \bn{n}{k}^{-1}
  = \sum_{k=0}^n p_k\,\xi^k
  = p(\xi). \qedhere
\]
\end{proof}

\begin{remark}
The apolar kernel
\[
  (1+\xi t)^n
\]
is a finite-dimensional analog of the Fischer--Fock kernel
\[
  e^{\xi t}
  =
  \sum_{k\ge 0}\frac{(\xi t)^k}{k!}.
\]
Indeed,
\[
  \left(1+\frac{\xi t}{n}\right)^n
  \longrightarrow
  e^{\xi t},
  \qquad n\to\infty,
\]
while for each fixed \(k\),
\[
  \frac{\binom{n}{k}}{n^k}
  \longrightarrow
  \frac{1}{k!}.
\]
Thus the apolar pairing may be viewed as a finite-level, or binomial, version of the Fischer--Fock inner product.
\end{remark}

\begin{defn}\label{def:beta}
For $p(t) = \sum_{k=0}^n p_k\, t^k \in \R_n[t]$, the
\emph{$\beta_n$-transform} is
\[
  \beta_n(p)_k := \frac{p_k}{\binom{n}{k}},
  \qquad k=0,1,\dots,n.
\]
Equivalently, $p(t) = \sum_{k=0}^n \bn{n}{k}\, \beta_n(p)_k\,t^k$.
\end{defn}

In these coordinates the apolar inner product takes the form
\[
  \ip{p}{q}_{\mathrm{ap}}
  = \sum_{k=0}^n \bn{n}{k}\,\beta_n(p)_k\,\beta_n(q)_k.
\]
We will see the full significance of the $\beta_n$-transform later, where $\beta_n(p)_k$ is exactly the common value of the multilinear coefficients of the multiaffine symmetrization lift of $p$ on sets of size~$k$.

\begin{defn}
Let
\[
  p(t)=\sum_{k=0}^n \bn{n}{k}\, \hat p_k t^k,
  \qquad
  q(t)=\sum_{k=0}^n \bn{n}{k}\, \hat q_k t^k.
\]

The \emph{finite free multiplicative convolution} of $p$ and $q$
is
\[
  p\boxtimes_n q
  :=
  \sum_{k=0}^n
  \bn{n}{k}\, \hat p_k\, \hat q_k\, t^k.
\]

The \emph{finite free additive convolution} of $p$ and $q$
is
\[
  p\boxplus_n q
  :=
  \sum_{k=0}^n
  \bn{n}{k}
  \left(
    \sum_{a+b=k}
    \bn{k}{a}\,
    \hat p_a\, \hat q_b
  \right)t^k.
\]
\end{defn}

\begin{remark}
The finite free convolutions are most naturally expressed in the
$\beta_n$-coordinates.  If
\[
  \beta_n(p)_k = \hat p_k,
  \qquad
  \beta_n(q)_k = \hat q_k,
\]
then multiplicative convolution is coefficientwise multiplication,
while additive convolution is binomial convolution:
\[
  \beta_n(p\boxtimes_n q)_k
  =
  \beta_n(p)_k\,\beta_n(q)_k,
\]
and
\[
  \beta_n(p\boxplus_n q)_m
  =
  \sum_{a+b=m}
  \bn{m}{a}\,
  \beta_n(p)_a\,\beta_n(q)_b.
\]

Several equivalent normalizations appear in the literature, especially in connection with characteristic polynomials and interlacing methods, where reversed indexing and alternating sign conventions are natural. Specifically, in \cite[Definitions 1.1 and 1.4]{MSS22} the finite free additive and multiplicative convolutions are defined as $(p^\star \boxplus_n q^\star)^\star$ and $(p^\star\boxtimes_n q^\star)^\star$, respectively, using our conventions.
\end{remark}

Let $\mathbb R_n[t_1,\dotsc,t_k]$ be the set of real multivariate polynomials of \emph{maximal} degree $n$ in each variable, so that $\mathbb R_n[t_1,\dotsc,t_k]\cong \mathbb R_n[t]^{\otimes k}$.
On $\mathbb R_n[s,t] \cong \R_n[s]\otimes\R_n[t]$ we use the tensor product apolar pairing,
which is given explicitly by
\[
  \ip{s^i t^j}{s^k t^\ell}_{\mathrm{ap}}
  = \bn{n}{k}^{-1}\bn{n}{\ell}^{-1}\,\delta_{ik}\,\delta_{j\ell}.
\]
We have a natural homomorphic embedding $\mathbb R_n[u]\hookrightarrow \mathbb R_n[s,t]$ given by $u\mapsto st$. We can push the apolar kernel forward under this map so that
\[
  K_n(\xi,st) = (1+\xi st)^n
  = \sum_{k=0}^n \bn{n}{k}\,\xi^k\, s^k t^k.
\]

\begin{prop}\label{prop:mult-conv-kernel}
For $p,q\in\R_n[t]$,
\[
  \ip{p(s)\,q(t)}{K_n(\xi,st)}_{\mathrm{ap}}
  = \sum_{k=0}^n \bn{n}{k}\,
  \beta_n(p)_k\,\beta_n(q)_k\,\xi^k = p\boxtimes_n q(\xi).
\]
\end{prop}

\begin{proof}
Since $K_n(\xi,st)$ is diagonal in $(s,t)$, expanding and
using orthogonality gives
\[
  \sum_{k=0}^n p_k\, q_k\, \bn{n}{k}\,\xi^k\,
  \bn{n}{k}^{-1}\bn{n}{k}^{-1}
  = \sum_{k=0}^n \binom{n}{k}\,
  \beta_n(p)_k\,\beta_n(q)_k\,\xi^k. \qedhere
\]
\end{proof}

Thus multiplicative convolution is coefficientwise
multiplication in $\beta_n$-coordinates.

Under the pushforward $u^k\mapsto (s+t)^k$ the kernel $K_n$ becomes
\[
  K_n(\xi,s+t) = (1+\xi(s+t))^n
  = \sum_{k=0}^n \bn{n}{k}\,\xi^k
  \sum_{a+b=k}\bn{k}{a}\, s^a t^b.
\]

\begin{prop}\label{prop:add-conv-kernel}
For $p,q\in\R_n[t]$,
\[
  \ip{p(s)\,q(t)}{K_n(\xi,s+t)}_{\mathrm{ap}}
  = \sum_{\substack{a,b\ge 0\\ a+b\le n}}
  p_a\, q_b\,
  \frac{\binom{a+b}{a}}{\binom{n}{a}\binom{n}{b}}\,
  \binom{n}{a+b}\,\xi^{a+b} = p\boxplus_n q(\xi).
\]
\end{prop}

\begin{proof}
Expanding $K_n(\xi,s+t)$ and using orthogonality gives the
result.
\end{proof}

\section{The Multilinear Model}
\label{sec:multilinear}

We now pass to a multilinear model in which both convolutions
become simultaneously transparent. In order to do so, we now define the \emph{squarefree algebra}.

Let $\mathcal{A}_n := \R[t_1,\dots,t_n]/(t_1^2,\dots,t_n^2)$
with monomial basis $\{t^S : S\subseteq [n]\}$, where $t^S := \prod_{i\in S} t_i$ and $t^\varnothing := 1$.
Multiplication is determined by
\begin{equation}\label{eq:mult-rule}
  t^A \cdot t^B =
  \begin{cases}
    t^{A\cup B}, & A\cap B = \varnothing,\\
    0,            & A\cap B \neq \varnothing.
  \end{cases}
\end{equation}
We equip $\mathcal{A}_n$ with the inner product making $\{t^S\}_{S\subseteq[n]}$ orthonormal.

\begin{prop}\label{prop:ml-kernel}
The reproducing kernel for $(\mathcal{A}_n, \ip{\cdot}{\cdot})$
is
\[
  K(\xi,t) = \prod_{i=1}^n (1+\xi_i t_i).
\]
That is, for every $f\in\mathcal{A}_n$,
\[
  f(\xi) = \ip{f(t)}{K(\xi,t)}.
\]
\end{prop}

\begin{proof}
Expanding the product,
\[
  K(\xi,t) = \sum_{S\subseteq[n]} \xi^S\, t^S.
\]
Orthonormality gives
$\ip{f(t)}{K(\xi,t)} = \sum_S f_S\,\xi^S = f(\xi)$.
\end{proof}

\bigskip

The symmetrization map
\[
  \Sym:\mathcal{A}_n\longrightarrow\R_n[t],
  \qquad t^S\mapsto t^{|S|},
\]
sends the multilinear kernel to the apolar kernel:
\[
  \Sym\Bigl(\prod_{i=1}^n(1+\xi\, t_i)\Bigr) = (1+\xi\, t)^n.
\]
Under this map, the $\beta_n$-coordinates of a polynomial are exactly its multilinear coefficients: if $f$ is symmetric with $f_S$ depending only on $|S|$, then $\beta_n(\Sym(f))_k = f_S$ for any $S$ with $|S|=k$.
Thus the $\beta_n$-transform identifies the univariate model with the symmetric subspace of $\mathcal{A}_n$.

\bigskip

Let $f,g: 2^{[n]}\to \mathbb R$ be set functions.

\begin{defn}\label{def:add-conv-ml}
The \emph{additive convolution} of $f,g$ is defined as
\[
  (f\boxplus g)_S
  := \sum_{A\subseteq S} f_A\, g_{S\setminus A}
  \qquad\text{for all }S\subseteq[n].
\]
\end{defn}

The following observation is central to everything that follows.

\begin{lem}\label{lem:boxplus-is-product}
Viewing $f,g$ as multiaffine polynomials, additive convolution is simply multiplication in
$\mathcal{A}_n$. That is, for all $f,g\in\mathcal{A}_n$,
\[
  f\boxplus g = fg \mod (t_1^2,\dots,t_n^2),
\]
with the evident abuse of notation of using equality to denote the canonical identification.
\end{lem}

\begin{proof}
By the multiplication rule~\eqref{eq:mult-rule}, the coefficient
of $t^S$ in $fg$ is
\[
  (fg)_S = \sum_{\substack{A,B\subseteq[n]\\
  A\cap B=\varnothing,\; A\cup B = S}} f_A\, g_B
  = \sum_{A\subseteq S} f_A\, g_{S\setminus A},
\]
since $A\cap(S\setminus A)=\varnothing$ automatically.
\end{proof}

For set functions $f,g: 2^{[n]}\to\mathbb R$ we also have the pointwise (Hadamard) product $(f\circ g)_S = f_S\, g_S$ for all $S\subseteq [n]$.

\begin{defn}\label{defn:multiaffine-lift}
    Given $p\in\mathbb R_n[t]$, we define the \emph{multiaffine lift} $\widetilde p$ of $p$ to be the unique symmetric multiaffine polynomial so that $\Sym(\widetilde p) = p$, that is, $\widetilde p_S = \beta_{n}(p)_{|S|}$ for all $S\subseteq [n]$.
\end{defn}

\begin{prop}\label{prop:sym-recovery}
Under symmetrization, the additive convolution and Hadamard product recover the univariate convolutions $\boxplus_n$ and $\boxtimes_n$, respectively. That is, 
\[
p\boxplus_n q = \Sym(\widetilde p \boxplus \widetilde q), \qquad p\boxtimes_n q = \Sym(\widetilde p \circ \widetilde q).
\]
\end{prop}

\begin{proof}
For symmetric $f$ and $g$ with
$f_S = \beta_n(p)_{|S|}$ and $g_S = \beta_n(q)_{|S|}$,
the multiplicative case is immediate from the definition and Proposition \ref{prop:mult-conv-kernel}.
For additive convolution, grouping by $|A|=j$ in $(f\boxplus g)_S = \sum_{A\subseteq S}f_A\,g_{S\setminus A}$ with $|S|=m$ gives $\binom{m}{j}$ subsets $A\subseteq S$ of size~$j$.  Since $f_A = \beta_n(p)_j$ depends only on $|A|$ (and similarly for $g$), the sum reduces to $\sum_{j=0}^m \binom{m}{j}\beta_n(p)_j\,\beta_n(q)_{m-j}$, recovering Proposition~\ref{prop:add-conv-kernel}.
\end{proof}

Thus, in the multilinear model, the two convolutions are
simultaneously simple: one is the Hadamard product, the other
is algebra multiplication.  No truncation occurs and no
binomial weights appear.

\begin{remark}
Multiplication in $\mathcal{A}_n$ is convolution with
respect to the Boolean lattice $(2^{[n]},\subseteq)$ for the ``diamond product'' defined in \cite{Sin26}. We do
not pursue a more systematic M\"obius-theoretic formulation here, as this will be detailed in forthcoming work.
\end{remark}

\bigskip

The following proposition is modeled closely after Propositions 4.19 and 4.20 in \cite{Choe04} and related arguments in \cite{Hin95}.

\begin{prop}\label{prop:real-stable-squarefree}
Let $p,q\in\mathbb R[t_1,\dots,t_n]$ be multiaffine real stable
polynomials. Then
\[
  pq \mod (t_1^2,\dots,t_n^2)
\]
is real stable.
\end{prop}

\begin{proof}
It suffices to show that deleting the $t_i^2$-coefficient from
a polynomial of degree at most two in $t_i$ preserves stability.
Write
\[
  F=A t_i^2+B t_i+C.
\]
Fix all variables other than $t_i$ in the upper half-plane. Then
$F$ becomes a one-variable quadratic with no zeros in the upper
half-plane. If its roots are $r,s\notin\mathbb H$, then the zero
of $B t_i+C$ is
\[
  -\frac{C}{B}=\frac{rs}{r+s}
  =\frac{1}{1/r+1/s},
\]
with the evident limiting interpretation in degenerate cases.
Since inversion sends the lower half-plane to the upper half-plane,
$1/r$ and $1/s$ lie in the closed upper half-plane, so
$1/r+1/s$ lies in the closed upper half-plane. Hence its reciprocal
does not lie in $\mathbb H$. Thus $B t_i+C$ has no zero in
$\mathbb H$.

Applying this operation successively to the stable polynomial
$pq$ deletes all square terms and yields
$pq \mod (t_1^2,\dots,t_n^2)$.
\end{proof}

\begin{lem}\label{lem:lift-apolar}
Let $p = \sum_{k=0}^n \bn{n}{k}\, \hat p_k\,t^k\in \mathbb C_n[t]$, and let $\widetilde p$ denote its symmetric multiaffine lift.
Then for every $\xi_1,\ldots,\xi_n\in \mathbb C$,
\[
\left\langle p(t),\prod_{j=1}^n(1+\xi_j t)\right\rangle_{\mathrm{ap}}
=
\widetilde p(\xi_1,\ldots,\xi_n).
\]
\end{lem}

\begin{proof}
Expanding the product gives
\[
    \prod_{j=1}^n(1+\xi_j t) =
    \sum_{k=0}^n e_k(\xi_1,\ldots,\xi_n)t^k,
\]
where $e_k(x_1,\dotsc,x_n)$ is the $k$-th elementary symmetric polynomial in $n$ variables.
From the definition of the apolar inner product it follows that
\[
\begin{aligned}
\left\langle p(t),\prod_{j=1}^n(1+\xi_j t)\right\rangle_{\mathrm{ap}}
&=
\sum_{k=0}^n
\bn{n}{k}^{-1}\bn{n}{k}\,\widehat p_k\, e_k(\xi_1,\ldots,\xi_n)
  \\
&=
\sum_{k=0}^n \widehat p_k\, e_k(\xi_1,\ldots,\xi_n) =
\widetilde p(\xi_1,\ldots,\xi_n).
\end{aligned}
\]
\end{proof}

The next proposition is Grace's theorem on root location for apolar pairs of polynomials written in the present reproducing-kernel framework. See \cite[Chapter IV, Section 15]{Marden} for the classical apolar formulation.

\begin{prop}\label{prop:apolar-root-location}
Let $C\subset \widehat{\mathbb C}$ be a circular region in the extended complex plane, and set $\iota$ to be the involution
\[
\iota(z):=-1/z.
\]
Suppose that $p\in\mathbb C_n[t]$ has all of its roots in $C$. If $q\in\mathbb C_n[t]$ satisfies
\[
\langle p,q\rangle_{\mathrm{ap}}=0,
\]
then $q$ has at least one root in $\iota(C)$.
\end{prop}

\begin{proof}
We prove the contrapositive in the case $q(0)\neq 0$. The general case follows from standard projective conventions or by a limiting argument.

Assume that $q$ has no zero in $\iota(C)$. We write
\[
q(t)=c\prod_{j=1}^n(1+\xi_jt),
\qquad c\neq 0, 
\]
where the roots of $q$ are $-1/\xi_j$. Thus by assumption we have that $\xi_j\not\in C$ for all $j=1,\dotsc,n$.

Let $\Omega=\widehat{\mathbb C}\setminus C$, interpreted as the complementary circular region. By the Grace--Walsh--Szeg\H{o} coincidence theorem (see \cite[Theorem (15,4)]{Marden} or \cite{BB09})  applied to the symmetric multiaffine lift $\widetilde p$, the nonvanishing of $p$ on $\Omega$ implies that
\[
\widetilde p(\xi_1,\ldots,\xi_n)\neq 0,
\]
hence, by the Lemma \ref{lem:lift-apolar},
\[
\langle p,q\rangle_{\mathrm{ap}}
=
c\,\widetilde p(\xi_1,\ldots,\xi_n)
\neq 0.
\]
This proves the contrapositive, hence the result.
\end{proof}

\begin{remark}
    The Grace--Walsh--Szeg\H{o} coincidence theorem can be derived by elementary means independently of Grace's theorem. They are in fact known to be equivalent: see the remarks following the proof of Theorem (15,4) in \cite{Marden}. Thus the proposition's effect is to mainly highlight the reproducing kernel structure as a natural bridge between the two results.
\end{remark}

As a corollary of the above discussion we offer a proof of stability of real rootedness under finite free additive convolution, which is essentially due to Walsh \cite{Wal22} and appears as \cite[Theorem (18,1)]{Marden}. See also \cite{MSS22}. Both this proof and the classical one are applications of the coincidence theorem.

\begin{cor}\label{cor:real-rooted-ffac}
    If $p,q\in \mathbb R_n[t]$ are real rooted, then so is $p\boxplus_n q$
\end{cor}

\begin{proof}
    From the Grace--Walsh--Szeg\H{o} theorem, we know that $p\in \mathbb R_n[t]$ is real rooted if and only if the multiaffine lift is real stable. The result follows by Propositions \ref{prop:sym-recovery} and \ref{prop:real-stable-squarefree}.
\end{proof}

\section{Cumulants and Linearization}
\label{sec:cumulants}

Every element of $\mathcal{A}_n$ with vanishing constant term is $n$-step nilpotent since a product of $n+1$ monomials of positive degree must repeat a variable.  For $f\in\mathcal{A}_n$ with $f_\varnothing=1$, define
\begin{equation}\label{eq:nilpotent-log}
  \log f := \sum_{k=1}^{n}
  \frac{(-1)^{k+1}}{k}\,(f-1)^k
  = \sum_{\varnothing\neq S\subseteq[n]}
  \kappa_S(f)\, t^S.
\end{equation}

\begin{defn}
    We refer to the coefficients $\kappa_S(f)$ as the \emph{squarefree cumulants} of $f$.
\end{defn}

Thus, if $f,g\in\mathcal{A}_n$ satisfy $f_\varnothing = g_\varnothing = 1$, then
\[
  \log(f\boxplus g) = \log f + \log g.
\]
Equivalently,
\[
  \kappa_S(f\boxplus g) = \kappa_S(f)+\kappa_S(g)
\]
for every nonempty $S\subseteq[n]$. Indeed, by Lemma~\ref{lem:boxplus-is-product}, $f\boxplus g = fg$.  Since the logarithm of any element $a$ with $a_\varnothing =0$ is defined by a terminating power series, the usual identity $\log(fg)=\log f+\log g$ holds.

\begin{prop} \label{prop:log-cumulant-formula}
Let
\[
  f(t)=\sum_{S\subseteq[n]} f_S\, t^S
  \in \mathcal{A}_n
\]
with $f_\varnothing=1$, and write
\[
  \log f
  =
  \sum_{\varnothing\neq S\subseteq[n]}
  \kappa_S(f)\, t^S.
\]
Then
\begin{equation}
      f_S = \sum_{\pi\in P(S)} \kappa_\pi(f),
\end{equation}
where
\[
  \kappa_\pi(f)
  :=
  \prod_{B\in\pi}\kappa_B(f),
\]
and $P(S)$ denotes the lattice of set partitions of~$S$.
Equivalently,
\begin{equation}
    \kappa_S(f)=
    \sum_{\pi\in P(S)} (-1)^{|\pi|-1}(|\pi|-1)! \prod_{B\in\pi} f_B.
\end{equation}
\end{prop}

\begin{proof}
Since $f=\exp(\log f)$ in the nilpotent algebra
$\mathcal{A}_n$,
\[
  f
  =
  \sum_{k=0}^n \frac{1}{k!}(\log f)^k.
\]
Fix $S\subseteq[n]$.  Expanding $(\log f)^k$, the coefficient
of $t^S$ is obtained by choosing an ordered decomposition
\[
  S=B_1\sqcup\cdots\sqcup B_k
\]
into nonempty disjoint subsets.  Since multiplication in $\mathcal{A}_n$ is squarefree, only disjoint products survive.
Grouping ordered decompositions according to the underlying
partition $\pi=\{B_1,\dots,B_m\}$ gives
\[
  f_S
  =
  \sum_{\pi\in P(S)}
  \prod_{B\in\pi}\kappa_B(f)
  =
  \sum_{\pi\in P(S)}
  \kappa_\pi(f).
\]

The inverse formula follows from M\"obius inversion on the
partition lattice: see \cite[Section 3.10]{Sta12}.
\end{proof}

The following is apparent from the natural action of the full permutation group acting on the generators $t_1,\dotsc,t_n$ of $\mathcal A_n$.

\begin{lem}\label{lem:symmetric-cumulants}
Let $p\in\R_n[t]$, and let $\widetilde p\in\mathcal{A}_n$ be its multiaffine lift.  Then $\kappa_S(\widetilde p)$ depends only on $|S|$.
\end{lem}

Accordingly, for the multiaffine lift we write
\[
  \kappa_k(p):=\kappa_S(\widetilde p),
  \qquad |S|=k.
\]

\begin{remark}\label{rem:AP-comparison}
Both Proposition~\ref{prop:log-cumulant-formula} and the construction of Arizmendi and Perales~\cite{AP18} are instances of the classical exponential formula \cite[equation (2.7)]{AP18}: a sequence $(c_k)$ and its classical cumulants $(b_k)$ satisfy
\[
  c_k=\sum_{\pi\in P([k])}\ \prod_{B\in\pi} b_{|B|}.
\]
The symmetric form of Proposition~\ref{prop:log-cumulant-formula} is this relation with $c_k=\beta_n(p)_k$ and $b_k=\kappa_k(p)$. The same exponential formula governs the finite free cumulants $\kappa^{\mathrm{AP}}_k$, applied to the normalized coefficient sequence
\[
  \widetilde a_k:=\frac{(-n)^{k}a_k}{(n)_k},
  \qquad (n)_k:=n(n-1)\dotsb(n-k+1),
\]
where $a_k$ are the coefficients of $p^\star$. This is exactly the sequence inside the logarithm of the finite $R$-transform \cite[equation (3.1)]{AP18}, with the partition--exponential extraction carried out in the proof of \cite[Proposition 3.4]{AP18}. The two coefficient systems differ only by the reversal and rescaling relating our normalization to that of \cite{AP18, MSS22}, so the squarefree cumulants and the finite free cumulants determine one another by an explicit rescaling. In this sense the finite free cumulants are, after rescaling, the ordinary logarithm in the squarefree algebra.
\end{remark}

\section{Multilinear log-concavity and closure under
\texorpdfstring{$\boxplus$}{additive convolution}}
\label{sec:logconcave}

The cumulant linearization of the preceding section suggests a
natural class defined by sign conditions on cumulants.

\begin{defn} \label{def:LC}
A polynomial $f\in\mathcal{A}_n$ with $f_S\ge 0$ for all $S$ and $f_\varnothing = 1$ is \emph{multilinearly log-concave}, written $f\in\mathcal{LC}_n$, if
\begin{equation}\label{eq:LC}
  \kappa_S(f)\le 0
  \qquad\text{for all } S\subseteq[n],\ |S|\ge 2.
\end{equation}
\end{defn}

\begin{remark}\label{rem:LC-interpretation}
For $|S|=2$, the condition $\kappa_{\{i,j\}}(f)\le 0$ is
equivalent to
\[
  f_{\{i,j\}} \le f_{\{i\}}\, f_{\{j\}},
\]
the usual pairwise negative correlation. Likewise,
$\kappa_S\le 0$ says that all higher-level correlations are similarly
repulsive.
\end{remark}

\bigskip

It is readily apparent that $\mathcal{LC}_n$ is closed under additive convolution, that is, if $f,g\in\mathcal{LC}_n$, then $f\boxplus g\in\mathcal{LC}_n$. Indeed, let $h = f\boxplus g$. Nonnegativity of $h_S$ follows from $f_A,g_B\ge 0$, and $h_\varnothing = f_\varnothing\, g_\varnothing = 1$.  For $|S|\ge 2$, cumulant additivity gives
\[
  \kappa_S(h) = \kappa_S(f) + \kappa_S(g) \le 0,
\]
since both terms are nonpositive by hypothesis.

\begin{remark}
    The simplicity of the present argument reflects the design of $\mathcal{LC}_n$: it is a large natural class for which cumulant additivity alone implies closure. On the other hand, the strongly Rayleigh property~\cite{BBL09} and the Lorentzian condition~\cite{BH20} are defined by nonlinear constraints, which reflect more subtle, ``global'' curvature properties of the polynomial. The precise containment relations between these classes remain to be understood.
\end{remark}

That said, the class $\mathcal{LC}_n$ enjoys several of the same formal permanence properties that appear in the theory of strongly Rayleigh measures and negative dependence. In fact, the class $\mathcal{LC}_n$ is closed under additive convolution, positive external fields, and marginalization, which are among the central permanence properties appearing in the theory of negatively dependent measures, \cite{Pem00,BBL09}.  The point here is not that $\mathcal{LC}_n$ coincides with the strongly Rayleigh class, but rather that the logarithmic coordinates naturally inherit many of the same structural symmetries.

\begin{prop}[External Fields/Projection]
\label{prop:LC-external-fields}
Let
\[
  f(t)=\sum_{S\subseteq[n]} f_S\, t^S
  \in \mathcal{LC}_n,
\]
and let $\lambda_1,\dots,\lambda_n \ge 0$.
Define
\[
  f^\lambda(t)
  :=
  f(\lambda_1 t_1,\dots,\lambda_n t_n).
\]
Then $f^\lambda\in\mathcal{LC}_n$.
\end{prop}

\begin{proof}
By Proposition \ref{prop:log-cumulant-formula}, specifically, the M\"obius inversion formula, we have that $\kappa_S$ is a sum of terms of the form $\prod_{B\in \pi} f_B$, where $\pi$ is a partition of $S$.
Applying the map
\[
  t_i\mapsto \lambda_i t_i,
\]
it follows that
\[
  \kappa_S(f^\lambda)
  =
  \lambda^S \kappa_S(f),
  \qquad
  \lambda^S:=\prod_{i\in S}\lambda_i.
\]
Since $\lambda^S \ge 0$, the inequalities $\kappa_S(f)\le 0$ for $|S|\ge 2$ are preserved.
\end{proof}

\begin{prop}[Marginalization]
\label{prop:LC-marginalization}
Let $f\in\mathcal{LC}_n$.  Then the marginal
\[
  g(t_1,\dots,t_{n-1})
  :=
  \frac{f(t_1,\dots,t_{n-1},1)}{1 + f_{\{n\}}}
\]
belongs to $\mathcal{LC}_{n-1}$.
\end{prop}

\begin{proof}
    Write
\[
  \log f = h_0+t_n h_1
\]
with \(h_0,h_1\in \mathcal A_{n-1}\). We have that
\[
    f = \exp(\log f) = \exp(h_0)\exp(t_nh_1) = \exp(h_0)(1 + t_n h_1)
\]
since $\mathcal A_n$ is commutative and $t_n^2=0$. Hence
\[
    \log g = - \log(1+f_{\{n\}}) + h_0 + \log(1 + h_1).
\]
Setting $\lambda_1,\dotsc,\lambda_{n-1}=1$ and $\lambda_n=0$, we have $\log f^\lambda = h_0$, so by Proposition \ref{prop:LC-external-fields}, we have $(h_0)_S\le 0$ for $|S| \ge 2$.

By the formula for the logarithm, the constant term of \(h_1\) is \(a := f_{\{n\}} = \kappa_{\{n\}}(f)\). Since \(f\in\mathcal{LC}_n\), we have $a\ge 0$ while all nonconstant coefficients of \(h_1\) are nonpositive. Thus
\[
  h_1 = a - h_1'
\]
where \(a\ge0\) and \(h_1'\in\mathcal A_{n-1}\) has nonnegative coefficients and vanishing constant term. Hence
\[
  \log(1+h_1)
  =
  \log(1+a) + \log\left(1-\frac{h_1'}{1+a}\right)
  =
  \log(1+a) - \sum_{m\ge1}\frac{1}{m}\left(\frac{h_1'}{1+a}\right)^m,
\]
Since \(h_1'\) has nonnegative coefficients, all nonconstant coefficients of the final sum are nonpositive.
Combining this with the nonpositivity of the higher-order
coefficients of \(h_0\), we conclude that every coefficient of
\(\log g = h_0 + \log\left(1 - \frac{h_1'}{1+a}\right)\) of degree at least two is nonpositive.  Since \(g\)
has nonnegative coefficients and \(g_\varnothing=1\), it follows
that \(g\in\mathcal{LC}_{n-1}\).
\end{proof}

\begin{prop}[Conditioning]
\label{prop:LC-conditioning}
Let \(f\in\mathcal{LC}_n\), and suppose \(f_{\{i\}}>0\).
Define the conditioned polynomial
\[
  f^{(i)}(t)
  :=
  \frac{1}{f_{\{i\}}}\,\frac{\partial f}{\partial t_i},
\]
viewed as an element of the squarefree algebra on the variables \(\{t_j:j\neq i\}\).  Then \(f^{(i)}\in\mathcal{LC}_{n-1}\).
\end{prop}

\begin{proof}
We may assume \(i=n\).  Write
\[
  \log f = h_0+t_n h_1,
\]
where \(h_0,h_1\in\mathcal A_{n-1}\).
Since \(f\in\mathcal{LC}_n\), the constant term of \(h_1\) is
\[
  a=\kappa_{\{n\}}(f)=f_{\{n\}}>0,
\]
while every nonconstant coefficient of \(h_1\) is nonpositive.
Thus
\[
  h_1=a - h_1',
\]
where \(h_1'\) has nonnegative coefficients and vanishing constant term.

Writing
\[
  f=\exp(h_0)(1+t_n h_1),
\]
we have
\[
  f^{(n)} =\frac{1}{f_{\{n\}}}\frac{\partial f}{\partial t_n}
  =
  \exp(h_0)\frac{h_1}{a}.
\]
Therefore
\[
  \log f^{(n)}
  =
  h_0+\log\left(\frac{h_1}{a}\right)
  =
  h_0+\log\left(1-\frac{h_1'}{a}\right).
\]
The proof now follows by the same reasoning as in the proof of Proposition \ref{prop:LC-marginalization}.
\end{proof}

\begin{prop}[Log--Submodularity]
\label{prop:LC-log-submodular}
Let \(f\in\mathcal{LC}_n\).  Then its coefficient function
\[
  S\longmapsto f_S
\]
is log-submodular on the Boolean lattice.  That is, for all
\(A,B\subseteq[n]\),
\[
  f_{A\cup B}\, f_{A\cap B} \le f_A\, f_B.
\]
\end{prop}

\begin{proof}
It suffices to prove the local form
\[
  f_{S\cup\{i,j\}}\, f_S
  \le
  f_{S\cup\{i\}}\, f_{S\cup\{j\}},
  \qquad i,j\notin S,\ i\neq j,
\]
since the local inequalities generate log-submodularity on the
Boolean lattice.

We may assume that \(f_S>0\) as the inequality is trivial otherwise.  Conditioning on \(S\) is defined by
\[
  f^{(S)}
  :=
  \frac{1}{f_S}\,\partial_S f,
\]
viewed as a polynomial in the remaining variables. The proof of Proposition~\ref{prop:LC-conditioning} is easily adapted to establish that \(f^{(S)}\in\mathcal{LC}_{n-|S|}\).  The \(\mathcal{LC}\) condition on the pair \(\{i,j\}\) gives
\[
  f^{(S)}_{\{i,j\}}
  \le
  f^{(S)}_{\{i\}}\, f^{(S)}_{\{j\}}.
\]
Since
\[
  f^{(S)}_{\{i,j\}}
  =
  \frac{f_{S\cup\{i,j\}}}{f_S},
  \qquad
  f^{(S)}_{\{i\}}
  =
  \frac{f_{S\cup\{i\}}}{f_S},
  \qquad
  f^{(S)}_{\{j\}}
  =
  \frac{f_{S\cup\{j\}}}{f_S},
\]
hence
\[
  f_{S\cup\{i,j\}}\, f_S
  \le
  f_{S\cup\{i\}}\, f_{S\cup\{j\}}.
\]
Therefore all local log-submodularity inequalities hold, hence \(S\mapsto f_S\) is log-submodular.
\end{proof}

\section{Trace--log expansion and examples}
\label{sec:laplacian}

We now apply the multilinear framework to determinantal examples.  In each case, the trace--log formula in $\mathcal{A}_n$ gives an explicit combinatorial description of the cumulants.

\subsection{The trace--log identity in $\mathcal{A}_n$}

Let $M\in M_n(\R)$ and set
\[
  D(t) := \diag(t_1,\dots,t_n),
  \qquad
  X := D(t)\, M \in M_n(\mathcal{A}_n).
\]
Since $t_i^2 = 0$ in $\mathcal{A}_n$, every entry of $X^m$ involves monomials of degree at least~$m$, so $X^{n+1} = 0$. Thus $X$ is nilpotent and $\log(I+X) = \sum_{m=1}^n \frac{(-1)^{m+1}}{m}\, X^m$ is well-defined.

\begin{prop}[Trace--log identity]
\label{prop:trace-log}
For every $M\in M_n(\R)$,
\[
  \log\det(I + D(t)\, M)
  = \Tr\log(I + D(t)\, M)
  = \sum_{m=1}^n \frac{(-1)^{m+1}}{m}\,
  \Tr\bigl((D(t)\, M)^m\bigr)
\]
in $\mathcal{A}_n$.
\end{prop}

\begin{proof}
The identity $\log\det(I+X) = \Tr\log(I+X)$ holds for nilpotent matrices over any commutative $\R$-algebra, as both sides agree as polynomial identities in the entries.
\end{proof}

\begin{prop}\label{prop:trace-cycle}
For each $m\ge 1$, the coefficient of $t^S$ (with $|S|=m$) in
$\Tr\bigl((D(t)\, M)^m\bigr)$ is
\[
  \sum_{\substack{(i_1,\dots,i_m)\\
  \{i_1,\dots,i_m\}=S}}
  M_{i_1 i_2}\, M_{i_2 i_3}\cdots M_{i_m i_1},
\]
where the sum is over all sequences $(i_1,\dots,i_m)$ that are
permutations of $S$.
\end{prop}

\begin{proof}
Expanding the matrix product and taking the trace, terms with
a repeated index vanish in $\mathcal{A}_n$ since $t_i^2=0$.
\end{proof}

Combining these gives the main formula.

\begin{thm}[Cycle expansion for multilinear cumulants]
\label{thm:cycle-expansion}
For $f_M(t) := \det(I+D(t)\, M)\in\mathcal{A}_n$ with
$\log f_M = \sum_{\varnothing\neq S} \kappa_S(M)\, t^S$,
\[
  \kappa_S(M)
  = (-1)^{k+1}
  \sum_{C(S)}
  M_{i_1 i_2}\, M_{i_2 i_3}\cdots M_{i_k i_1},
\]
where $k=|S|$ and $C(S)$ denotes classes of cyclic orderings of $S$ modulo rotation.
\end{thm}

\begin{proof}
The monomial $t^S$ arises only from the $m=|S|$ term.  Each
cyclic ordering is counted $k$ times by starting point, absorbing
the factor $1/k$.
\end{proof}

\begin{remark}\label{rem:small-cumulants}
For small sets, the formula gives:
\begin{itemize}
\item $|S|=1$:
$\kappa_{\{i\}}(M) = M_{ii}$.

\item $|S|=2$:
$\kappa_{\{i,j\}}(M) = -M_{ij}\, M_{ji}$.

\item $|S|=3$:
$\kappa_{\{i,j,k\}}(M) = M_{ij}\, M_{jk}\, M_{ki}
+ M_{ik}\, M_{kj}\, M_{ji}$.
\end{itemize}
The singleton cumulant records the diagonal entry, the pairwise
cumulant records (minus) the product of the off-diagonal pair,
and the triple cumulant records the two directed triangles
through the three vertices.
\end{remark}

\subsection{Graph Laplacians}

Let $G$ be a finite simple graph on $[n]$ with Laplacian $L_G$,
and set $f_G(t) := \det(I + D(t)\, L_G)$.

\begin{thm}[Hamiltonian-cycle formula]
\label{thm:hamiltonian}
Write $\log f_G = \sum_{\varnothing\neq S} \kappa_S(G)\, t^S$.
Then:
\begin{enumerate}
  \item[\textup{(i)}]
  $\kappa_{\{i\}}(G) = \deg_G(i)$.
  \item[\textup{(ii)}]
  For $|S|\ge 2$,
  \[
    \kappa_S(G) = -\,h(G[S]),
  \]
  where $h(G[S])$ is the number of oriented Hamiltonian cycles
  in the induced subgraph $G[S]$, modulo rotation.
  In particular, for $|S|=2$: $\kappa_{\{i,j\}}(G) = -1$ if
  $ij\in E(G)$, and $0$ otherwise, since a Hamiltonian cycle on
  two vertices is a single directed $2$-cycle.
\end{enumerate}
In particular, $\kappa_S(G)\le 0$ for all $|S|\ge 2$.
\end{thm}

\begin{proof}
Since the off-diagonal entries of $L_G$ are $0$ or $-1$, an ordering of $S$ contributes to Theorem~\ref{thm:cycle-expansion} if and only if every consecutive pair is an edge---that is, if and only if it determines an oriented Hamiltonian cycle in $G[S]$.  When it does, the product of $k$ factors of $-1$ gives $(-1)^k$, so $(-1)^{k+1}\cdot(-1)^k = -1$.
\end{proof}

\begin{cor}\label{cor:laplacian-LC}
For every finite graph $G$,
\[
  \det(I + \diag(t)\, L_G) \in \mathcal{LC}_n.
\]
\end{cor}

\begin{proof}
The polynomial $f_G$ has nonnegative coefficients and $f_G(0)=1$.  By Theorem~\ref{thm:hamiltonian}, $\kappa_S(G)\le 0$ for $|S|\ge 2$.
\end{proof}

\begin{remark}\label{rem:graph-examples}
The trace--log formula becomes combinatorially literal: the logarithm isolates connected structures, and the cumulants are supported precisely on those induced subgraphs admitting Hamiltonian cycles.
\end{remark}

\begin{example}[Trees]
     No induced subgraph on $\ge 3$ vertices contains a Hamiltonian cycle, so
\[
  \log f_G(t) = \sum_{i\in V}\deg_G(i)\, t_i
  - \sum_{ij\in E} t_i t_j.
\]
\end{example}

\begin{example}[Cycles]
    For $C_m$, the full vertex set admits exactly two oriented Hamiltonian cycles, so $\kappa_{[m]}(C_m) = -2$. All other higher cumulants vanish.
\end{example}

\begin{example}[Complete graphs]  
    Every cyclic ordering of $[m]$ is a Hamiltonian cycle, so $\kappa_{[m]}(K_m) = -(m-1)!$.  For $S\subseteq[m]$ with $|S|=k\ge 2$,\; $\kappa_S(K_m) = -(k-1)!$.
\end{example}

\subsection{Edge Laplacians}
\label{ssec:edge-laplacian}

Let $G=(V,E)$ be a finite simple graph with $V=[d]$.  For each edge $e=\{i,j\}\in E$, fix an arbitrary orientation and set $b_e := e_i - e_j$ and $C_e := b_e b_e^\top$.  (Since $C_e = b_e b_e^\top$ is unchanged by reversing the orientation of~$e$, the matrix $T$ and hence all cumulants are independent of the choice of orientation.)  In the squarefree algebra $R_E := \R[t_e : e\in E]/(t_e^2)$, set $T := \sum_{e\in E} t_e\, C_e$ and expand
\[
  -\log\det(I - T)
  = \sum_{k\ge 1}\frac{1}{k}\,\Tr(T^k).
\]
For $S\subseteq E$ with $|S|=k$, the edge cumulant is
\[
  \kappa_S
  = \frac{1}{k}\sum_{\sigma\in\mathfrak{S}_k}
  \prod_{j=1}^k \ip{b_{e_{\sigma(j)}}}{b_{e_{\sigma(j+1)}}},
\]
where the indices are cyclic.  Although the individual inner products $\ip{b_e}{b_f}$ depend on orientations, each edge vector appears exactly twice in any cyclic product (once as the ``outgoing'' factor and once as ``incoming''), so reversing the orientation of any edge flips two signs and leaves the product invariant.  Thus $\kappa_S$ is well-defined.

\begin{prop}\label{prop:bipartite}
If $G$ is bipartite, then $\kappa_S \ge 0$ for all
$S\subseteq E$.
\end{prop}

\begin{proof}
The bipartite orientation ensures every inner product $\ip{b_e}{b_f}$ is nonnegative: shared vertices in the same part contribute $+1$, and disjoint edges contribute~$0$.  Each cyclic product is therefore nonnegative.
\end{proof}

\begin{remark}\label{rem:bipartite-necessary}
The converse also holds.  Suppose $G$ contains an odd cycle $C = (v_1,v_2,\dots,v_{2\ell+1})$, and let $S = \{v_1v_2, v_2v_3,\dots,v_{2\ell+1}v_1\}$ be its edge set. Orient the edges cyclically, so that each edge is directed $v_i \to v_{i+1}$ (indices modulo $2\ell+1$).  Then $b_{v_i v_{i+1}} = e_{v_i} - e_{v_{i+1}}$, and for each consecutive pair one computes
\[
  \ip{b_{v_i v_{i+1}}}{b_{v_{i+1} v_{i+2}}} = -1.
\]
Thus the cyclic product is $(-1)^{2\ell+1} = -1$, and $\kappa_S = \frac{1}{k}(-1)(2k) = -2< 0$.

Therefore $G$ is bipartite if and only if all edge cumulants are nonnegative.
\end{remark}


\appendix

\section{Projective geometries and the \(q\)-additive coefficients}

In this appendix we record a small observation showing that the squarefree picture from Section~\ref{sec:multilinear} has a projective-geometric analog. This connects the diamond product from \cite{Sin26} with the $q$-additive finite free convolution recently considered by Mart\'inez-Finkelshtein, Morales, and Perales~\cite{MFMP26}.  The point is that the coefficients of the latter arise naturally from the rank-radial part of the diamond product on projective geometries.

Throughout this appendix, $q$ denotes a prime power.  Thus $q$ is the size of the finite field considered and \emph{not} the analytic $q$-parameter of basic hypergeometric series. Let
\[
    L=\PG(n-1,q)
\]
be the lattice of linear subspaces of $\mathbb F_q^n$, ordered by inclusion.  As in \cite{Sin26}, the diamond product on $\R[L]$ is
defined on lattice basis vectors by
\[
    e_x\diamond e_y
    =
    \begin{cases}
        e_{x\vee y}, & x\wedge y=\hat 0,\\
        0, & x\wedge y\neq \hat 0,
    \end{cases}
\]
where $\hat 0$ denotes the minimal element of $L$. We refer the reader to \cite{Sin26} for additional details and discussion.

\begin{remark}
The Boolean lattice case recovers the squarefree algebra appearing in Section~\ref{sec:multilinear}.  Indeed, if \(L=2^{[n]}\), then the diamond product is disjoint union, hence $\mathbb R[L]$ identifies with
\[
    \mathbb R[t_1,\dotsc,t_n]/(t_1^2,\dotsc,t_n^2)
\]
by sending a subset $S\subseteq[n]$ to \(t^S\).  Under this identification, rank-radial coefficients are precisely the \(\beta_n\)-coordinates used above.
The projective calculation below is therefore the corresponding rank-radial calculation for the next homogeneous geometric lattice, namely the lattice of subspaces of \(\mathbb F_q^n\).
\end{remark}

For $0\leq k\leq n$, set
\[
    L_k:=\{x\in L: \dim(x)=k\},
    \qquad
    R_k:=\sum_{x\in L_k}e_x.
\]
Then by homogeneity the span of $R_0,\dotsc,R_n$ is closed under the diamond product.

\begin{prop}\label{prop:projective-radial-diamond}
For \(a+b=k\), one has
\[
    R_a\diamond R_b
    =
    q^{ab}\,\binom{k}{a}_q R_k,
\]
where $\bn{n}{k}_q$ is the usual $q$-binomial coefficient.
Consequently, if
\[
    f=\sum_{k=0}^n \widehat f_k R_k,
    \qquad
    g=\sum_{k=0}^n \widehat g_k R_k,
\]
then
\[
    f\diamond g =
    \sum_{k=0}^n
    \left(
        \sum_{a+b=k}
        q^{ab}\,\binom{k}{a}_q
        \widehat f_a\,\widehat g_b
    \right)R_k.
\]
\end{prop}

\begin{proof}
Fix a $k$-dimensional subspace $z\leq \mathbb F_q^n$.  The coefficient of $e_z$ in $R_a\diamond R_b$ is the number of ordered decompositions
\[
    z=x\oplus y,\qquad \dim(x)=a,\quad \dim(y)=b.
\]
There are $\bn{k}{a}_q$ choices for $x\leq z$, and for each such $x$ there are $q^{ab}$ complements $y$ to $x$ in $z$. Hence the coefficient of every $e_z$ with $\dim(z)=k$ is $q^{ab}\binom{k}{a}_q$, proving the claim.
\end{proof}

Thus the natural rank-radial coefficient system for the projective diamond algebra is 
\[ 
    (\widehat f,\widehat g) \longmapsto \widehat h, \qquad \widehat h_k = \sum_{a+b=k} q^{ab}\binom{k}{a}_q \widehat f_a\,\widehat g_b. 
\] 
This is the direct analog of the Boolean formula 
\[ 
    \widehat h_k = \sum_{a+b=k} \binom{k}{a} \widehat f_a\,\widehat g_b,
\] 
which is the rank-radial form of the Boolean diamond product, i.e., squarefree multiplication. To compare this with the $q$-additive finite free convolution, one passes from rank-radial coordinates to total rank-mass coordinates. Set
\[ 
    \gamma^q_k(f):=\binom{n}{k}_q \widehat f_k. 
\] 
Since $\binom{n}{k}_q=|L_k|$, the number $\gamma^q_k(f)$ is the total coefficient mass of $f$ on the rank $k$ layer. If \(h=f\diamond g\), then 
\[ 
    \gamma_k^q(h) = \sum_{a+b=k} q^{ab} \frac{ \binom{n}{k}_q\binom{k}{a}_q }{ \binom{n}{a}_q\binom{n}{b}_q }\, \gamma^q_a(f)\,\gamma^q_b(g). 
\] 
Using 
\[ 
    \bn{n}{k}_q = \frac{(q;q)_n}{(q;q)_k(q;q)_{n-k}}, 
\] 
this becomes 
\[ 
    \gamma_k^q(h) = \sum_{a+b=k} q^{ab} \frac{ (q;q)_{n-a}(q;q)_{n-b} }{ (q;q)_n(q;q)_{n-k} }\, \gamma^q_a(f)\, \gamma^q_b(g). 
\] 

Now put $Q=q^{-1}$, where $Q$ is the analytic parameter used in the $Q$-hypergeometric convention. Since $k=a+b$, the elementary identity \[ 
    \frac{ (Q;Q)_{n-a}(Q;Q)_{n-b} }{ (Q;Q)_n(Q;Q)_{n-k} } = 
    q^{ab} \frac{ (q;q)_{n-a}(q;q)_{n-b} }{ (q;q)_n(q;q)_{n-k} } 
\] 
gives 
\[
    \gamma_k^q(h) = \sum_{a+b=k} \frac{ (Q;Q)_{n-a}(Q;Q)_{n-b} }{ (Q;Q)_n(Q;Q)_{n-k} } \gamma^q_a(f)\, \gamma^q_b(g). 
\] 
This is exactly the coefficient factor appearing in the $Q$-additive finite free convolution. 
The additional factor $(-1)^k$ in the usual definition of $\boxplus_{n,Q}$ comes from the standard characteristic-polynomial coefficient convention. 

Thus, after the standard reversal and sign convention, the \(Q\)-additive finite free convolution with $Q=q^{-1}$ is the total rank-mass form of the rank-radial projective diamond product.

\begin{remark} 
    The $Q$-multiplicative convolution defined by Lamprecht \cite{Lam16} admits a parallel interpretation at the level of the radial kernel. The identity element for $\boxtimes_{n,Q}$ is 
    \[ 
    e_{n,Q}(x) :=\prod_{j=1}^n(1-Q^{-j}x). 
    \] 
    When \(Q=q^{-1}\), this becomes 
    \[ 
        e_{n,q^{-1}}(x)=\prod_{j=1}^n(1-q^j x), 
    \] 
    which is naturally related to the characteristic polynomial of the projective geometry up to a change of variable. Thus the $Q$-multiplicative convolution should be viewed as the corresponding Hadamard operation in the M\"obius-normalized radial kernel. We do not use this observation here. 
\end{remark}


\section*{Acknowledgments}
The author thanks Daniel Perales for helpful comments on the relation between Corollary~\ref{cor:real-rooted-ffac} and the classical Grace--Walsh--Szeg\H{o} theory, and for pointing to Marden’s treatment of apolar polynomials.
The author used OpenAI’s ChatGPT for editorial assistance and for exploring examples and formulations during the development of this work. All mathematical content, results, and conclusions are solely those of the author.


\end{document}